\documentclass[12pt]{amsart} % later change to 10pt maybe?

\makeatletter
\newcommand\mathcircled[1]{%
	\mathpalette\@mathcircled{#1}%
}
\newcommand\@mathcircled[2]{%
	\tikz[baseline=(math.base)] \node[draw,ellipse,inner sep=1pt] (math) {$\m@th#1#2$};%
}

\usepackage{multicol}
\usepackage{latexsym}
\usepackage{psfrag}
\usepackage{amsmath}
\usepackage{amssymb}
\usepackage{epsfig}
\usepackage{amsfonts}
\usepackage{amscd}
\usepackage{mathrsfs}
\usepackage{graphicx}
\usepackage{enumerate}
\usepackage[autostyle=false, style=english]{csquotes}
\MakeOuterQuote{"}
\usepackage{ragged2e}
\usepackage[all]{xy}
\usepackage{mathtools}

\usepackage{nicefrac}

\newlength\ubwidth

\usepackage[dvipsnames]{xcolor}

\usepackage{placeins}
\usepackage{tikz}
\usetikzlibrary{shapes, positioning}
\usetikzlibrary{matrix,shapes.geometric,calc,backgrounds}
\usetikzlibrary{arrows.meta} 
\usetikzlibrary{shapes.geometric, arrows, positioning, decorations.pathreplacing} % Added decorations.pathreplacing
\newcommand\bsfrac[2]{%
\scalebox{-1}[1]{\nicefrac{\scalebox{-1}[1]{$#1$}}{\scalebox{-1}[1]{$#2$}}}%
}

\newtheorem{theorem}{Theorem}%[section]
\newtheorem{lemma}[theorem]{Lemma}

\newtheorem{conj}[theorem]{Conjecture}

\usepackage{amsmath,amssymb,amsthm}
\usepackage[dvipsnames]{xcolor}
\usepackage[colorlinks=true,citecolor=RoyalBlue,linkcolor=red,breaklinks=true]{hyperref}

\usepackage{graphicx}
\usepackage{mathdots} % for anti diagonal dots
\usepackage[margin=2.5cm]{geometry}
\usepackage{ytableau}

\usepackage{cancel}

\usepackage{blkarray}
\usepackage{booktabs}

\begin{document}%%%%%%%%%%%%%%%%%%%%%%%%%%%%%%%%%%%%%%%%%%%%%%%%%%%%%%%%
	%%%%%%%%%%%%%%%%%%%%%%%%%%%%%%%%%%%%%%%%%%%%%%%%%%%%%%%%%%%%%%%%%%%%%%%%

\title[generating $k$-regular partitions]
{An Alternative Generating Function for $k$-Regular Partitions}

\author[Kur\c{s}ung\"{o}z]{Ka\u{g}an Kur\c{s}ung\"{o}z}
\address{Ka\u{g}an Kur\c{s}ung\"{o}z, Faculty of Engineering and Natural Sciences, 
    Sabanc{\i} University, Tuzla, Istanbul 34956, Turkey}
\email{kursungoz@sabanciuniv.edu}


      
\date{2025}

\begin{abstract}
 We construct a $k$-fold $q$-series as a generating function of 
 $k$-regular partitions for each positive integer $k$.  
 The $k=1$ case is one of Euler's $q$-series identities 
 pertaining to the partitions into distinct parts.  
 The construction is combinatorial.  
 Although we find a connection to Bessel polynomials in the $k=2$ case, 
 this note is certainly not a study of Bessel polynomials and their $q$-analogs.  
\end{abstract}

\maketitle

An integer partition of a non-negative integer $n$ 
is a non-decreasing sequence of positive integers 
whose sum is $n$~\cite{theBlueBook}.  
For example, 4 has the following five partitions.  
\begin{align}
\nonumber 
  4, \quad 
  2+2, \quad 
  1+3, \quad  
  1+1+2, \quad 
  1+1+1+1.  
\end{align}
Among these, $4$ and $1+3$ are partitions into distinct parts.  
If we denote the number of partitions of $n$ satisfying condition $A$ 
by $p(n \vert A)$, 
then 
\begin{align}
\nonumber 
  \sum_{n \geq 0} p(n \vert A) q^n
\end{align}
is called a partition generating function.  
It is well known~\cite{theBlueBook} that 
\begin{align}
\nonumber 
  \sum_{n \geq 0} p(n) q^n = \frac{1}{(q; q)_\infty}, 
  \quad \textrm{ and } \quad 
  \sum_{n \geq 0} p(n \vert \textrm{ distinct parts }) q^n = (-q; q)_\infty, 
\end{align}
where 
\begin{align}
\nonumber 
  (a; q)_n = \prod_{j = 1}^n \left( 1 - a q^{j-1} \right), 
  \quad \textrm{ and } \quad 
  (a; q)_\infty = \lim_{n \to \infty} (a; q)_n.  
\end{align}
The infinite products (in fact, any series or infinite product in this paper) 
converge absolutely when $ \vert q \vert < 1$~\cite{theBlueBook}.  

It is often necessary to work with a second parameter 
and consider the number of partitions of $n$ into $m$ parts 
satisfying condition $A$, denoted by $p(m, n \vert A)$.  
Again, it is well known~\cite{theBlueBook} that 
\begin{align}
\nonumber 
  \sum_{m, n \geq 0} p(m, n) x^m q^n = \frac{1}{(xq; q)_\infty}, 
  \quad \textrm{ and } \quad 
  \sum_{m, n \geq 0} p(m, n \vert \textrm{ distinct parts }) x^m q^n = (-xq; q)_\infty.  
\end{align}

$k$-regular partitions 
are those in which no part is repeated more than $k$ times.  
1-regular partitions are partitions into distinct parts.  
The seven 2-regular partitions of 6 are 
\begin{align}
\nonumber 
  6, \quad 
  1+5, \quad 
  2+4, \quad 
  1+1+4, \quad 
  3+3, \quad 
  1+2+3, \quad 
  1+1+2+2.  
\end{align}
Depending on the purpose, 
it is possible to represent partitions in different ways.  
For instance, we can drop the plus signs for brevity, 
or stack the repeated parts on top of each other 
for emphasis on $k$-regularity.  
So, the last displayed partition may be written as 
$1 \; 1 \; 2 \; 2$, or 
$\begin{array}{cc} 1 & 2 \\ 1 & 2\end{array}$.  

$k$-regular partitions are widely studied 
either for their combinatorial properties (e.g.~\cite{BallMercRegPtn}), 
or their arithmetic properties ( e.g.~\cite{BallMercRegPtn, BeckBesskRegPtn}).  
Our main goal is to prove the following theorem (Theorem \ref{thm2RegMain})
for 2-regular partitions, and then to generalize it to 
$k-$regular partitions for any positive integer $k$ (Theorem \ref{thmkReg}).  
We will also note a connection between the polynomials $b(m, n)$ 
in Theorem \ref{thm2RegMain} and the coefficient of Bessel polynomials~\cite{GrossBess, BesselSurvey}.  
However, our aim is not studying the Bessel polynomials here.  

\begin{theorem}
\label{thm2RegMain}
 \begin{align}
 \nonumber 
  \frac{ ( x^3 q^3; q^3 )_\infty }{ ( xq; q )_\infty } 
  = \sum_{ m, n \geq 0 } \frac{ q^{ \binom{m+n+1}{2} 
      + \binom{m+1}{2} } (1 - q)^m x^{2m+n} }{ (q; q)_{2m+n} }
    b(m, n), 
 \end{align}
 where $b(0,0) = 1$, $b(m, n) = 0$ if $m<0$ or $n<0$, and 
 \begin{align}
 \nonumber 
  b(m, n) = \left( 1 + q + q^2 + \cdots + q^{2m+n-2} \right) b(m-1, n)
    + q^m b(m, n-1).  
 \end{align}
\end{theorem}

It is well known~\cite{theBlueBook} that 
\begin{align}
\label{genfuncDist}
 \sum_{n \geq 0} \frac{ q^{ \binom{n+1}{2} } }{ (q; q)_n }
\end{align}
generates partitions into distinct parts.  
For any non-negative integer $n$, 
we can think of $q^{ \binom{n+1}{2} }$ as generating 
a base partition into exactly $n$ parts with the smallest weight, 
that is $1 + 2 + \cdots + n$.  
Then, for the same $n$, $\frac{ 1 }{ (q; q)_n }$ 
generates partitions into $n$ parts, zeros allowed.  
Take $0 \leq \lambda_1 \leq \lambda_2 \leq \cdots \leq \lambda_n$ as a generic one.  
Finally, their combination given by 
\begin{align}
\nonumber 
 (1 + \lambda_1) + (2 + \lambda_2) + \cdots + (n + \lambda_n)
\end{align}
generates a partition into exactly $n$ distinct parts.  
Conversely, given a partition into $n$ distinct parts 
$\mu_1 + \mu_2 + \cdots + \mu_n$ where $0 < \mu_1 < \mu_2 < \cdots < \mu_n$, 
we have 
\begin{align}
\nonumber 
 0 \leq \mu_1 - 1 \leq \mu_2 - 2 \leq \cdots \leq \mu_n - n.  
\end{align}
Setting $\lambda_j = \mu_j - j$ for $j = 1, 2, \ldots, n$ 
will give us a partition $0 \leq \lambda_1 \leq \lambda_2 \leq \cdots \leq \lambda_n$ 
into $n$ parts, zeros allowed, 
along with the base partition into $n$ parts with minimal weight $1 + 2 + \cdots + n$.  

We can also imagine parts of the base partition moving forward of backward.  
For example, part $j$ in the base partition moves forward $\lambda_j$ times 
for $j = n, (n-1), \ldots, 1$, in this order.  
{\allowdisplaybreaks
\begin{align}
\nonumber 
 & 1 \quad 2 \quad \cdots \quad (n-1) \quad \mathbf{n} \qquad 
 (\textrm{ the highlighted part moves forward } \lambda_n \textrm{ times }) \\ 
\nonumber 
 & \rightsquigarrow 1 \quad 2 \quad \cdots \quad \mathbf{(n-1)} \quad (n + \lambda_n) \qquad 
 (\textrm{ the highlighted part moves forward } \lambda_{n-1} \textrm{ times }) \\ 
\nonumber 
 & \qquad \qquad \vdots \\ 
\nonumber 
 & (1 + \lambda_1) \quad (2 + \lambda_2) \quad \cdots \quad (n + \lambda_n) \qquad
 = \quad \mu_1 \quad \mu_2 \quad \cdots \quad \mu_n
\end{align}}
For any non-negative integer $n$, 
this gives a bijection between partitions into $n$ distinct parts 
($0 < \mu_1 < \mu_2 < \cdots < \mu_n$) 
and pairs consisting of the base partition $1 + 2 + \cdots + n$ 
and a partition into $n$ parts 
$0 \leq \lambda_1 \leq \lambda_2 \leq \cdots \leq \lambda_n$, zeros allowed.  
This is the most direct combinatorial explanation of the fact that 
\eqref{genfuncDist} generates partitions into distinct parts.  
We can also consider 
\begin{align}
\nonumber 
 \sum_{n \geq 0} \frac{ q^{ \binom{n+1}{2} } x^n }{ (q; q)_n }.  
\end{align}
The same bijection works.  
The exponent of $x$ keeps track of the number of parts, 
while the exponent of $q$ is the number being partitioned.  

For our purposes, we will let $\frac{ 1 }{ (q; q)_n }$ generate 
the conjugate of $\lambda_1 + \lambda_2 + \cdots + \lambda_n$.  
This is a partition into parts which are at most $n$, 
with $f_1$ 1's, $f_2$ 2's, \ldots, $f_n$ $n$'s.  
In this case, each $j$ moves the $j$ largest parts in the intermediate partition 
forward, one each, $f_j$ times.  
This is done for each $j = 1, 2, \ldots, n$.  
By the intermediate partition, 
we mean any partition encountered in constructing any partition 
$\mu_1 + \mu_2 + \cdots + \mu_n$ into distinct parts 
from the base partition $1 + 2 + \cdots + n$ with these steps.  

We want to imitate the above bijection for 2-regular partitions, 
that is the partitions into parts that repeat at most twice.  
This was attempted in~\cite{KOevidPos, HalimeThesis} also, 
while studying 2-regular partitions that arose naturally 
in a class of partitions due to Kanade and Russell~\cite{KRstair}.  
We will discuss the differences between approaches after the construction.  

\begin{lemma}
\label{lemmaForMainThm}
 \begin{align}
 \nonumber 
  & \frac{ (x^3 q^3; q^3)_\infty }{ ( xq; q)_\infty }
  = \prod_{ n \geq 1 } \left( 1 + x q^n + x^2 q^{2n} \right) \\ 
 \nonumber 
  = & \sum_{ \substack{ m, n \geq 0 \\ 1 \leq i_1 < i_2 < \cdots < i_m \leq m+n} }
      \frac{ q^{ \binom{m+n+1}{2} + i_1 + i_2 + \cdots + i_m } 
         x^{ 2m+n } }{ (q; q)_{2m+n} } \cdots \\
 \nonumber 
  & \qquad \qquad \times ( 1 - q^{ 2m+n + 1 - i_1 - 1 }) 
  ( 1 - q^{ 2m+n + 1 - i_2 - 2 }) \cdots ( 1 - q^{ 2m+n + 1 - i_m - m })
 \end{align}
\end{lemma}

\begin{proof}
  It is clear that the infinite products in the lemma 
  generate 2-regular partitions.  
  The exponent of $q$ keeps track of the weight, 
  and the exponent of $x$ keeps track of the length, 
  that is the number of parts.  
  
  Given a 2-regular partition, 
  we will move parts backward so as to minimize the weight, 
  and record the number of moves in an auxiliary partition
  (c.f.~\cite{KParity, K_AG}).  
  However, we will keep the repeated parts (hereafter \emph{pairs}) 
  as well as the relative positions of pairs and non-repeated parts (hereafter \emph{singletons}).  
  Let $\mu_1$, $\mu_2$, \ldots, $\mu_{m+n}$ be the part sizes of an 
  arbitrary but fixed 2-regular partition having $m$ pairs and $n$ singletons 
  such that 
  \begin{align}
  \nonumber 
  0 < \mu_1 < \mu_2 < \cdots < \mu_{m+n}.  
  \end{align}
  This 2-regular partition has $2m+n$ parts, 
  because exactly $m$ of the $\mu_j$'s repeat by assumption.  
  Call them $\mu_{i_1}$, $\mu_{i_2}$, \ldots $\mu_{i_m}$, 
  where 
  \begin{align}
  \nonumber 
  1 \leq i_1 < i_2 < \cdots < i_m \leq m+n.  
  \end{align}
  Observe that the 2-regular partition at hand 
  has $2m + n - i_j - j$ parts which are $> \mu_{i_j}$ 
  and $2m + n + 2 - i_j - j$ parts which are $\geq \mu_{i_j}$ 
  for each $j =$ 1, 2, \ldots, $m$.  
  \begin{align}
  \label{ptn2RegPartCount}
  (\textrm{ parts } < \mu_{i_j}) \qquad 
  \begin{array}{c} \mu_{i_j} \\ \mu_{i_j} \end{array} \qquad 
  \underbrace{ 
    \overbrace{ (\textrm{ parts } > \mu_{i_j}) }^{ (m-j) \textrm{ of them repeat } }
  }_{ (m+n-i_j) \textrm{ part sizes } }
  \end{align}
  Unless the 2-regular partition is the empty partition (generated by $q^0 = $ 1), 
  we run the following procedure.  
  \begin{itemize}
  \item[{\bf 1b - }] Set $k = 1$, 
    the auxiliary partition $\lambda = $ the empty partition $\varepsilon$, 
    and $r = $ the number of parts that are $\geq \mu_k$.  
  \item[{\bf 2b - }] Subtract $(\mu_k - k)$ from all parts that are $\geq \mu_k$, 
    and append $(\mu_k - k)$ $r$'s to the auxiliary partition $\lambda$.  
    (Part sizes $\mu_j$'s are dynamically updated here, 
    $\mu$ is the intermediate partition.)  
  \item[{\bf 3b - }] Increment $k$ by 1.  If $k > m+n$, then stop.  
  \end{itemize}
  Now we have the base partition with part sizes 1, 2, \ldots $(m+n)$, 
  in which $i_1$, $i_2$, \ldots, $i_m$ repeat, namely
  \begin{align}
  \nonumber 
  \begin{array}{cccc cccc ccc}
      &   &        & i_1 &           &        &           & i_2 &           &        & \\ 
    1 & 2 & \cdots & i_1 & (i_1 + 1) & \cdots & (i_2 - 1) & i_2 & (i_2 + 1) & \cdots & (m+n)
  \end{array}, 
  \end{align}
  and an auxiliary partition $\lambda$ into parts that are at most $(2m+n)$.  
  By the observation just above the procedure, $\lambda$ cannot contain 
  parts that are equal to 
  \begin{align}
  \nonumber 
  (2m+n + 1 - i_1 - 1), \quad 
  (2m+n + 1 - i_2 - 2), \quad \cdots \quad 
  (2m+n + 1 - i_m - m).  
  \end{align}
  The base partition has weight $\binom{m+n+1}{2} + i_1 + i_2 + \cdots + i_m$.  
  Therefore, the base partition and the auxiliary partition pairs are generated by 
  \begin{align}
  \label{termGenPairs} 
  \frac{ q^{ \binom{m+n+1}{2} + i_1 + i_2 + \cdots + i_m } 
    ( 1 - q^{ 2m+n + 1 - i_1 - 1 }) ( 1 - q^{ 2m+n + 1 - i_2 - 2 }) 
    \cdots ( 1 - q^{ 2m+n + 1 - i_m - m })}{ (q; q)_{2m+n} }.  
  \end{align}
  The procedure 1b-3b also preserves the relative placement of the singletons and the pairs.  

  Conversely, suppose that we are given the base partition with part sizes 1, 2, \ldots, $(m+n)$ 
  in which $i_1$, $i_2$, \ldots, $i_m$ occur twice ($1 \leq i_1 < i_2 < \cdots < i_m \leq m+n$), 
  and all others occur once; 
  and a partition $\lambda$ into parts that are at most $(2m+n)$ 
  having no parts equal to $(2m+n+1 - i_1 - 1)$, $(2m+n+1 - i_2 - 2)$, \ldots $(2m+n+1 - i_m - m)$.  
  First; similar to \eqref{ptn2RegPartCount}, 
  we can argue that the base partition has $2m+n-i_j-j$ parts 
  which are strictly greater than $i_j$, 
  and $2m+n+2-i_j-j$ parts which are at least $i_j$
  for $j = $ 1, 2, \ldots, $m$.  
  \begin{align}
  \label{basePtn2RegPartCount}
  \begin{array}{cccc} & & & i_j \\ 1 & 2 & \cdots & i_j \end{array} 
  \underbrace{\overbrace{\begin{array}{ccccccc} & & & i_{j+1} & & & \\ 
    (i_j + 1) & \cdots & (i_{j+1}-1) & i_{j+1} & (i_{j+1}-1) & \cdots & (m+n) \end{array}}^{
    (m-j) \textrm{ of them repeat } }}_{ (m+n-i_j) \textrm{ part sizes } }
  \end{align}
  We run the following procedure.  
  \begin{itemize}
  \item[{\bf 1f - }] Set $r = 1$ and the intermediate partition $\mu = $ 
    the base partition in \eqref{basePtn2RegPartCount}.  
  \item[{\bf 2f - }] Let $f_r$ be the number of parts in $\lambda$ 
    that are equal to $r$.  
    Increment the largest $r$ parts in $\mu$ $f_r$ times.  
  \item[{\bf 3f - }] Increment $r$ by 1.  
    If $r > 2m+n$, then stop.  
  \end{itemize}
  Now we have a 2-regular partition $\mu$ 
  which has $(m+n)$ part sizes $m$ of which repeat.  
  By the observation just above \eqref{basePtn2RegPartCount}, 
  the base partition and the constructed $\mu$ 
  have the same relative ordering of singletons and pairs.  
  Again, by the same observation, 
  there is no ambiguity in applying step 2f.  
  This is because some part sizes in $\lambda$ are restricted, 
  so that the largest $r$ parts never contain exactly one copy of a repeated part.  
  The base partition has weight $\binom{m+n+1}{2} + i_1 + \cdots + i_m$, 
  so that the base partition and auxiliary partition pairs 
  are generated by \eqref{termGenPairs}.  
\end{proof}

Let us give an example to the backward phase of the proof of Lemma \ref{lemmaForMainThm}.  
Consider the following 2-regular partition.  
\begin{align}
\nonumber 
  \begin{array}{ccccc} & & 10 & & 19 \\ 3 & 6 & 10 & 15 & 19 \end{array}
\end{align}
The part sizes are 
\begin{align}
\nonumber 
  \mu_1 = 3 \quad < \quad 
  \mu_2 = 6 \quad < \quad 
  \mu_3 = 10 \quad < \quad 
  \mu_4 = 15 \quad < \quad 
  \mu_5 = 19.  
\end{align}
We set $k = 1$, $\lambda = \varepsilon$, the empty partition.  
There are 7 parts greater than or equal to $\mu_1$.  
We subtract $\mu_1 - 1 = 2$ from all of them, 
and append 2 7's to $\lambda$.  
Now, the 2-regular partition becomes 
\begin{align}
\nonumber 
  \begin{array}{ccccc} & & 8 & & 17 \\ 1 & 4 & 8 & 13 & 17 \end{array}, 
\end{align}
and $\lambda = 7 \; 7$.  
We increment $k$, it becomes $k = 2$.  
There are 6 parts $\geq \mu_2 = 4$.  
We subtract $\mu_2 - 2 = 2$ from all of them, 
and append 2 6's to $\lambda$.  
The 2-regular partition becomes 
\begin{align}
\nonumber 
  \begin{array}{ccccc} & & 6 & & 15 \\ 1 & 2 & 6 & 11 & 15 \end{array}, 
\end{align}
and $\lambda = 6 \; 6 \; 7 \; 7$.  
The same operation for $k = 3$ makes the intermediate 2-regular partition 
\begin{align}
\nonumber 
  \begin{array}{ccccc} & & 3 & & 12 \\ 1 & 2 & 3 & 8 & 12 \end{array}, 
\end{align}
and $\lambda = 5 \; 5 \; 5 \; 6 \; 6 \; 7 \; 7$.  
Because $\mu_3$ is a repeated part, $\lambda$ will skip a part.  
It will not have any 4's.  
In the end, the 2-regular partition will be the base partition 
\begin{align}
\nonumber 
  \begin{array}{ccccc} & & 3 & & 5 \\ 1 & 2 & 3 & 4 & 5 \end{array}, 
\end{align}
and $\lambda = 2 \; 2 \; 2 \; 3 \; 3 \; 3 \; 3 \; 5 \; 5 \; 5 \; 6 \; 6 \; 7 \; 7$.  
$\lambda$ will not have any 1's, either.  
We invite the reader to run the forward phase of the algorithm 
in the proof of Lemma \ref{lemmaForMainThm} 
on the last displayed 2-regular base partition and $\lambda$.  

The combinatorial approach in Lemma \ref{lemmaForMainThm} is more straightforward 
than in~\cite{KOevidPos, HalimeThesis}.  
In contrast, the index range in the multiple sum in Lemma \ref{lemmaForMainThm} 
is more intricate than the multiple sum 
that is constructed in~\cite{KOevidPos, HalimeThesis}, 
due to more relaxed conditions on the base partition.  
Another difference is that the construction in~\cite{KOevidPos, HalimeThesis} 
works for 2-regular partitions only.  

The series in Lemma \ref{lemmaForMainThm} has evidently positive coefficients, too.  
After the proof of Lemma \ref{lemmaForMainThm}, 
the proof of Theorem \ref{thm2RegMain} becomes 
a straightforward $q$-series manipulation.  

%%%%%%%%%%%%%%%%%%%%%%%%%%%%%%%%%%%%%%%%%%%%%%%%%%%%%%%%%%%%%
%% begin proof of Theorem \ref{thm2RegMain}
%%%%%%%%%%%%%%%%%%%%%%%%%%%%%%%%%%%%%%%%%%%%%%%%%%%%%%%%%%%%%
\begin{proof}[proof of Theorem \ref{thm2RegMain}]
  Rewrite the series in Lemma \ref{lemmaForMainThm} as
  \begin{align}
  \nonumber 
   \sum_{m, n \geq 0} \frac{ q^{ \binom{m+n+1}{2} } x^{2m+n} }{ (q; q)_{2m+n} } 
   \sum_{ 1 \leq i_1 < \cdots < i_m \leq m+n } q^{ i_1 + \cdots + i_m } 
    \left( 1 - q^{ 2m+n+1-i_1-1 } \right) \cdots \left( 1 - q^{ 2m+n+1-i_m-m } \right), 
  \end{align}
  and call the inner sum $a(m, n)$.  
  We can calculate 
  {\allowdisplaybreaks \begin{align}
  \nonumber 
   a(m, n) & = q ( 1 - q^{2m+n-1} )
   \sum_{ 1 < i_2 < \cdots < i_m \leq m+n } q^{ i_2 + \cdots + i_m } 
    \left( 1 - q^{ 2m+n+1-i_2-2 } \right) \cdots \left( 1 - q^{ 2m+n+1-i_m-m } \right) \\ 
  \nonumber 
   & + \sum_{ 1 < i_1 < \cdots < i_m \leq m+n } q^{ i_1 + \cdots + i_m } 
    \left( 1 - q^{ 2m+n+1-i_1-1 } \right) \cdots \left( 1 - q^{ 2m+n+1-i_m-m } \right), 
  \end{align}}
  where the separated sums are for the cases $i_1 = 1$ and $i_1 > 1$, respectively.  
  {\allowdisplaybreaks \begin{align}
  \nonumber 
   & = q^m ( 1 - q^{2m+n-1} )
   \sum_{ 1 \leq i_2 - 1 < \cdots < i_m - 1 \leq (m-1)+n } q^{ (i_2 - 1) + \cdots + (i_m - 1) } 
    \left( 1 - q^{ 2(m-1)+n+1-(i_2-1)-1 } \right) \cdots \\ 
  \nonumber 
      & \hspace{8cm} \times \left( 1 - q^{ 2(m-1)+n+1-(i_m-1)-(m-1) } \right) \\ 
  \nonumber 
   & + q^m \; \sum_{ 1 \leq i_1-1 < \cdots < i_m-1 \leq m+(n-1) } q^{ (i_1-1) + \cdots + (i_m-1) } 
    \left( 1 - q^{ 2m+(n-1)+1-(i_1-1)-1 } \right) \cdots \\ 
  \nonumber 
      & \hspace{6cm} \times \left( 1 - q^{ 2m+(n-1)+1-(i_m-1)-m } \right)  \\ 
  \label{fneqForAs}
   & = q^m ( 1 - q^{2m+n-1} ) a(m-1, n) + q^m a(m, n-1)
  \end{align}}
  The recurrence \eqref{fneqForAs} uniquely defines $a(m, n)$ 
  for integers $m$ and $n$ once we declare that 
  $a(0,0) = 1$ and $a(m, n) = 0$ for $m<0$ or $n<0$.  
  These initial conditions are also apparent in Lemma \ref{lemmaForMainThm}.  
  The proof follows once we set 
  \begin{align}
  \nonumber 
    a(m, n) = q^{ \binom{m+1}{2} } (1 - q)^m b(m, n)
  \end{align}
  for all integers $m$ and $n$.    
\end{proof}
%%%%%%%%%%%%%%%%%%%%%%%%%%%%%%%%%%%%%%%%%%%%%%%%%%%%%%%%%%%%%
%% end proof of Theorem \ref{thm2RegMain}
%%%%%%%%%%%%%%%%%%%%%%%%%%%%%%%%%%%%%%%%%%%%%%%%%%%%%%%%%%%%%

For example, $a(0, n) = 1$ for all $n \geq 0$ and 
\begin{align}
\nonumber 
  a(m, 0) = q^{\binom{m+1}{2}} (-q; q^2)_m
\end{align}
for all $m \geq 0$.  
Values of $a(m, n)$ for small $m$ and $n$, 
in particular for $m+n \leq 4$, 
are shown in the following table.  

\begin{tabular}{|p{5mm}|p{40mm}|p{35mm}|p{35mm}|p{20mm}|p{5mm}|}
 \hline
 $\bsfrac{n}{m}$ & 0 & 1 & 2 & 3 & 4 \\ 
 \hline
 0 & 1 & 1 & 1 & 1 & 1 \\
 \hline
 1 & $q $ $- q^2$ & $q$ $+ q^2$ $- 2q^3$ & $q$ $+ q^2$ $+ q^3$ $- 3 q^4$ 
  & $q$ $+ q^2$ $+ q^3$ $ + q^4$ $- 4 q^5$ & \\
 \hline
 2 & $ q^3$ $- q^4$ $- q^6$ $+ q^7 $ & $ q^3$ $+ q^4$ $- q^5$ $- q^6 $   
    $- q^7$ $-2 q^8$ $+ 3 q^9$ & 
  $ q^{3}$ $+ q^{4}$ $+ 2 q^{5}$ $- 2 q^{6} $   
    $ - q^{7}$ $- 2 q^{8}$ $- 2 q^{9}$ $- 3 q^{10}$ $+ 6 q^{11} $ & 
   & \\
 \hline
 3 & $ q^{6}$ $- q^{7}$ $- q^{9}$ $+ q^{10} $  $ - q^{11}$ 
  $+ q^{12}$ $+ q^{14}$ $- q^{15} $ & 
  $q^{6}$ $+q^{7}$ $-q^{8}$ $-2q^{10}$ $-2q^{11}$ 
  $+q^{12}$ $+2q^{15}$ $+q^{16}$ $+3q^{17}$ $-4q^{18}$ & & & \\
 \hline
 4 & $q^{10}$ $-q^{11}$ $-q^{13}$ $+q^{14}$ $-q^{15}$ 
  $+q^{16}$ $-q^{17}$ $+2q^{18}$ $-q^{19}$ $+q^{20}$ 
  $-q^{21}$ $+q^{22}$ $-q^{23}$ $-q^{25}$ $+q^{26}$ & & & & \\
 \hline
\end{tabular}

Values of $b(m, n)$ for small $m$ and $n$, 
in particular, for $m+n \leq 4$
are shown in the following table.  

\begin{tabular}{|p{5mm}|p{40mm}|p{35mm}|p{35mm}|p{20mm}|p{5mm}|}
 \hline
 $\bsfrac{n}{m}$ & 0 & 1 & 2 & 3 & 4 \\ 
 \hline
 0 & 1 & 1 & 1 & 1 & 1 \\
 \hline
 1 & 1 & $1$ $+2q$ & $1$ $+2q$ $+3q^2$ & $1$ $+2q$ $+3q^2$ $+4q^3$ & \\
 \hline
 2 & $1$ $+q$ $+q^2$ & $1$ $+3q$ $+4q^{2}$ $+4q^{3}$ $+3q^{4}$ & 
  $1$ $+3q$ $+7q^{2}$ $+9q^{3}$ $+10q^{4}$
  $+9q^{5}$ $+6q^{6}$ & & \\
 \hline
 3 & $1$ $+2q$ $+3q^{2}$ $+3q^{3}$ $+3q^{4}$
  $+2q^{5}$ $+q^{6}$ & 
  $1$ $+4q$ $+8q^{2}$ $+13q^{3}$ $+17q^{4}$
  $+18q^{5}$ $+17q^{6}$ $+14q^{7}$ $+9q^{8}$
  $+4q^{9}$ & & & \\
 \hline
 4 & $1$ $+3q$ $+6q^{2}$ $+9q^{3}$ $+12q^{4}$
  $+14q^{5}$ $+15q^{6}$ $+14q^{7}$ $+12q^{8}$
  $+9q^{9}$ $+6q^{10}$ $+3q^{11}$ $+q^{12}$
  & & & & \\
 \hline
\end{tabular}

Observe that the polynomials in the latter table are all unimodal.  
A unimodal polynomial $p(q) = c_0 + c_1 q + \cdots +c_r q^r$ 
has coefficients satisfying 
\begin{align}
\nonumber 
  c_0 \leq c_1 \leq \cdots c_{s-1} \leq c_s \geq c_{s+1} \geq \cdots \geq c_r
\end{align}
for some $0 \leq s \leq r$.  

Letting $q \to 1$ makes $a(m, n)$'s all 1's and 0's.  
However, letting $q \to 1$ makes $b(m, n)$'s coefficients of 
Bessel polynomials~\cite{GrossBess, BesselSurvey}, 
because the closed formula
\begin{align}
\nonumber 
  \lim_{q \to 1} b(m, n) = \frac{ (2m+n)! }{ m! (m+n)! 2^m}
\end{align}
satisfies the recurrence in Theorem \ref{thm2RegMain} upon $q \to 1$, 
as well as the initial values.  
These are the coefficients of the Bessel polynomials~\cite{GrossBess, BesselSurvey}
\begin{align}
\nonumber 
  y_n(x) = \sum_{k = 0}^n \frac{ (n+k)! }{n! k!} \left( \frac{x}{2} \right)^k
\end{align}
after the change of parameters $n \leftarrow m+n$ and $k \leftarrow m$.  
However, the $q$-analog of the coefficients of Bessel polynomials~\cite{BesselSurvey}
\begin{align}
\nonumber 
  y_{n, q}(x) = \sum_{k = 0}^n \frac{ [n+k]_q ! }{ [n]_q ! [k]_q!} 
    \left( \frac{x}{2} \right)^k
\end{align}
are different from $b(m, n)$ in Theorem \ref{thm2RegMain}, 
because they do not satisfy the recurrence in the theorem
under any change of parameters.  
Here, 
\begin{align}
\nonumber 
  [n]_q = 1 + q + \cdots + q^{n-1}, 
  \quad \textrm{ and } \quad 
  [n]_q ! = [n]_q [n-1]_q \cdots [1]_q.  
\end{align}
In this sense, $b(m, n)$'s are another $q$-analog of the 
coefficients of Bessel polynomials.  
We failed to find a closed formula for $b(m, n)$'s.  

It is routine to adapt the machinery 
in Lemma \ref{lemmaForMainThm} and Theorem \ref{thm2RegMain} 
for 3-regular partitions.  
One finds that 
\begin{align}
\nonumber 
% \label{id3Reg}
  \frac{ ( x^4 q^4; q^4 )_\infty }{ ( xq; q)_\infty } 
  = \sum_{l, m, n \geq 0} \frac{ q^{ \binom{l+m+n+1}{2} + \binom{l+m+1}{2} 
    + \binom{l+1}{2} } (1 - q)^{2l+m} x^{3l+2m+n} }{ ( q; q)_{3l+2m+n} } 
    b(l, m, n)
\end{align}
where 
$b(0,0,0) = 1$, $b(l,m,n)= 0 $ if $l<0$ or $m<0$ or $n<0$, and 
\begin{align}
\nonumber 
  b(l,m,n) & = \left( 1 + q + q^2 + \cdots + q^{3l+2m+n-2} \right) 
    \left( 1 + q + q^2 + \cdots + q^{3l+2m+n-3} \right) b(l-1, m, n) \\ 
\nonumber 
  & + \left( 1 + q + q^2 + \cdots + q^{3l+2m+n-2} \right) q^l b(l, m-1, n) \\ 
\nonumber 
  & + q^{2l+m} b(l, m, n-1).  
\end{align}

In fact, we can go even further, 
and do it for $k$-regular partitions for any integer $k \geq 1$.  

\begin{theorem}
\label{thmkReg}
 For any positive integer $k$, 
 {\allowdisplaybreaks \begin{align}
 \nonumber 
  & \frac{ \left( x^{k+1} q^{k+1}; q^{k+1} \right)_\infty }{ ( x q; q )_\infty } \\ 
 \nonumber 
  = & \sum_{ n_k, n_{k-1}, \ldots, n_1 \geq 0 } 
   \frac{ q^{ \binom{n_k + n_{k-1} + \cdots + n_1 + 1}{2} 
        + \binom{n_k + n_{k-1} + \cdots + n_2 + 1}{2}
        + \cdots \binom{n_k+1}{2} } 
     (1 - q)^{ (k-1)n_k + (k-2)n_{k-1} + \cdots 2 n_3 + n_2 } }
    { (q; q)_{kn_k + (k-1)n_{k-1} + \cdots + 2n_2 + n_1} } \cdots \\[5pt]
 \nonumber 
  & \qquad \qquad \times \; x^{kn_k + (k-1)n_{k-1} + \cdots + 2n_2 + n_1 } 
  b(n_k, n_{k-1}, \ldots, n_1), 
 \end{align}}
 where $b(0, \ldots, 0) = 1$, 
 $b(n_k, n_{k-1}, \ldots, n_1) = 0$ if $n_j < 0$ for any $j = $1, 2, \ldots, $k$, 
 and 
 {\allowdisplaybreaks \begin{align}
 \nonumber 
  & b(n_k, n_{k-1}, \ldots, n_1) \\ 
 \nonumber 
  & = \left( 1 + q + q^2 + \cdots + q^{ k n_k + (k-1) n_{k-1} 
    + \cdots + 2 n_2 + n_1 - 2 } \right) \cdots \\ 
 \nonumber 
  & \times \left( 1 + q + q^2 + \cdots + q^{ k n_k + (k-1) n_{k-1} 
    + \cdots + 2 n_2 + n_1 - 3 } \right) \cdots \\ 
 \nonumber 
  & \vdots \\ 
 \nonumber 
  & \times \left( 1 + q + q^2 + \cdots + q^{ k n_k + (k-1) n_{k-1} 
    + \cdots + 2 n_2 + n_1 - k } \right) 
    b( n_k - 1 , n_{k-1}, n_{k-2}, \ldots, n_1) \\[10pt]
 \nonumber 
  & + \cdots \\[10pt]
 \nonumber 
  & + \left( 1 + q + q^2 + \cdots + q^{ k n_k + (k-1) n_{k-1} 
    + \cdots + 2 n_2 + n_1 - 2 } \right) \cdots \\ 
 \nonumber 
  & \times \left( 1 + q + q^2 + \cdots + q^{ k n_k + (k-1) n_{k-1} 
    + \cdots + 2 n_2 + n_1 - 3 } \right) \cdots \\ 
 \nonumber 
  & \vdots \\ 
 \nonumber 
  & \times \left( 1 + q + q^2 + \cdots + q^{ k n_k + (k-1) n_{k-1} 
    + \cdots + 2 n_2 + n_1 - j } \right) \cdots \\ 
 \nonumber 
  & q^{ (k-j) n_k + (k-1-j) n_{k-1} + \cdots + 2 n_{j+2} + n_{j+1} }
    b( n_k, \ldots, n_{j+1}, n_j - 1, n_{j-1}, \ldots, n_1) \\[10pt] 
 \nonumber 
  & + \cdots \\[10pt] 
 \nonumber 
  & + q^{ (k-1) n_k + (k-2) n_{k-1} + \cdots + 2 n_3 + n_2 }
    b( n_k, \ldots, n_2, n_1 - 1).  
 \end{align}}
\end{theorem}

The proof is a straightforward extension of the proofs of 
Lemma \ref{lemmaForMainThm} and Theorem \ref{thm2RegMain}, 
and it is omitted.  
Theorem \ref{thmkReg} is Euler's formula~\cite{theBlueBook} for $k = 1$, 
and Theorem \ref{thm2RegMain} for $k = 2$.  
$b(n_1) = 1$ for all $n_1 \geq 0$.  
Upon $q \to 1$, $b(n_2, n_1)$ are the coefficients of Bessel polynomials~\cite{GrossBess, BesselSurvey}, 
as discussed above.  
However; for $q \to 1$, values of $b(n_3, n_2, n_1)$ do not show up in 
OEIS, the online encyclopedia of integer sequences~\cite{oeis}.  
% The recurrence in Theorem \ref{thmkReg} does not reduce 
% to the recurrence given for extensions of the coefficients of Bessel polynomials 
% as given in (citey).  
% Still, it will not be surprising if $\lim_{q \to 1} b(n_k, \ldots, n_1)$ 
% is a linear combination of the extension of the coefficients 
% of the Bessel polynomials as given in (citey).  
Empirical evidence suggests that 
$b(n_k, \ldots, n_1)$, which are polynomials in $q$, are always unimodal.  
We leave it as a conjecture here.  

\begin{conj}
\label{conjKreg} 
 For any integer $k \geq 2$, 
 and any non-negative $n_k$, $n_{k-1}$, \ldots $n_1$, 
 $b(n_k, \ldots, n_1)$ constructed in Theorem \ref{thmkReg} 
 is unimodal.  
\end{conj}

Please notice that this is essentially different from 
unimodality properties of the coefficients of the Bessel polynomials, 
e.g.~\cite{ChoiSmitUnim}.  

% \newpage

\bibliographystyle{amsplain}

\begin{thebibliography}{10}

\bibitem{theBlueBook} Andrews, G.E., 1998. 
\emph{The theory of partitions} (No. 2). Cambridge university press. 

\bibitem{BallMercRegPtn} Ballantine, C. and Merca, M., 2023. 
6-regular partitions: new combinatorial properties, congruences, and linear inequalities. 
\emph{Revista de la Real Academia de Ciencias Exactas, Físicas y Naturales. Serie A. Matemáticas}, 
{\bf 117}(4), p.159. 

\bibitem{BeckBesskRegPtn} Beckwith, O. and Bessenrodt, C., 2016. 
Multiplicative properties of the number of k-regular partitions. 
\emph{Annals of Combinatorics}, {\bf 20}, pp.231--250. 

\bibitem{ChoiSmitUnim} Choi, J.Y. and Smith, J.D., 2003. 
On the unimodality and combinatorics of Bessel numbers. 
\emph{Discrete Mathematics}, {\bf 264}(1--3), pp.45--53.

\bibitem{GrossBess} Grosswald, E., 2006. 
\emph{Bessel polynomials} (Vol. 698). Springer.

\bibitem{KRstair} Kanade, S. and Russell, M.C., 2019. 
Staircases to Analytic Sum-Sides for Many New Integer Partition Identities of Rogers-Ramanujan Type. 
\emph{The Electronic Journal of Combinatorics}, pp.P1--6.

\bibitem{KParity} Kurşungöz, K., 2010. 
Parity considerations in Andrews–Gordon identities. 
\emph{European Journal of Combinatorics}, {\bf 31}(3), pp.976--1000. 

\bibitem{K_AG} Kurşungöz, K., 2019. 
Andrews–Gordon type series for Capparelli's and Göllnitz–Gordon identities. 
\emph{Journal of Combinatorial Theory, Series A}, {\bf 165}, pp.117--138.

\bibitem{KOevidPos} Kurşungöz, K. and Ömrüuzun Seyrek, H., 2022. 
Construction of evidently positive series and an alternative construction 
for a family of partition generating functions due to Kanade and Russell. 
\emph{Annals of Combinatorics}, {\bf 26}(4), pp.903--942.

\bibitem{oeis} OEIS, 
The On-Line Encyclopedia of Integer Sequences, 
available at \url{https://oeis.org/}.  

\bibitem{HalimeThesis} Ömrüuzun Seyrek, H., 2021. 
\emph{Construction of evidently positive series and an alternative construction 
for a family of partition generating functions due to Kanade and Russell} 
(Doctoral dissertation, Sabanc{\i} University).

\bibitem{BesselSurvey} Srivastava, H.M., 2023. 
An introductory overview of Bessel polynomials, 
the generalized Bessel polynomials and the q-Bessel polynomials. 
\emph{Symmetry}, {\bf 15}(4), p.822.

\end{thebibliography}

\end{document}